\documentclass[12pt,reqno]{amsart}
\usepackage{amsmath}
 \usepackage[latin1]{inputenc}
\usepackage{amssymb , amsmath, amsthm}

\usepackage{latexsym}
\usepackage[dvips]{graphicx}
\usepackage{ulem}

\usepackage{color}

\newcommand{\R}{\mathbb R}

\newcommand{\C}{\mathbb C}

\newtheorem{theorem}{Theorem}[section]

\newtheorem{definition}{Definition}[section]

\newtheorem{proposition}{Proposition}[section]
\newtheorem{remark}{Remark}[section]

\title[Placement of an obstacle ]
{On the placement of an obstacle
so as to optimize the
Dirichlet heat trace}

\author[Ahmad El Soufi and Evans M. Harrell II]{}

\keywords{ Dirichlet Laplacian, eigenvalues, heat trace, determinant, obstacle, spherical shell.}
\subjclass[2000]{35P15, 49R50, 58J50}

\email{elsoufi@univ-tours.fr}
\email{harrell@math.gatech.edu}

\begin{document}
\maketitle

\centerline{\scshape Ahmad El Soufi }
\medskip
{\footnotesize
  \centerline{Laboratoire de Math\'ematiques et Physique Th\'eorique,
UMR CNRS 7350} 
   \centerline{Universit\'e Fran\c{c}ois Rabelais de Tours, Parc de Grandmont, F-37200
Tours, France}}

\medskip

\centerline{\scshape Evans M. Harrell II}
\medskip
{\footnotesize
  \centerline{School of Mathematics}
\centerline{Georgia Institute of Technology, 
Atlanta GA 30332-0160, USA} }
   
\bigskip

\begin{abstract}
We prove that among all doubly connected domains of $\R^n$ bounded by two spheres of given radii, 
$Z(t)$, the trace of the heat kernel with Dirichlet boundary conditions, achieves its minimum when the spheres are concentric (i.e., for the spherical shell).  
The supremum is attained when the interior sphere is in contact with the outer sphere.

This is shown to be a special case of a more general 
theorem characterizing the optimal placement of a spherical obstacle inside a 
convex domain so as to maximize or minimize
the trace of the Dirichlet heat kernel.  In this case the minimizing position of the center of the obstacle belongs to  
the ``heart'' of the domain,
while the maximizing situation occurs 
either in the interior of the heart or
at a point where the obstacle is in contact with the outer boundary. 

Similar statements hold for the optimal positions of the obstacle
for any spectral property that can be obtained as 
a positivity-preserving or positivity-reversing transform of $Z(t)$,
including the spectral zeta function and, through it, the regularized determinant.

\end{abstract}

\section {Introduction and statement of results}\label{1}
Let $\Omega\subset \R^n$ be a bounded 
$C^2$
Euclidean domain and let 
$$\lambda_1(\Omega)  <  \lambda_2(\Omega)  \le \lambda_3(\Omega) \le 
 \cdots \le \lambda_i(\Omega) \le \cdots\rightarrow \infty ,$$
be the sequence of eigenvalues of the Dirichlet realization of the Laplacian $-\Delta$ in $\Omega$, where each eigenvalue is repeated according to its multiplicity. The corresponding ``heat operator" $e^{t\Delta}$ has finite trace for all $t>0$ (known in physical literature as the ``partition function''), 
which we denote
\begin{equation}\label{Zdef}
 Z_{\Omega}(t) = \sum_{k\ge 1} e^{- \lambda_k(\Omega) t}.
\end{equation}
Let $\zeta_{\Omega}$ be the zeta function,
defined as the meromorphic extension to the entire complex plane of 
 $ \sum_{k=1}^{\infty}\lambda_k(\Omega)^{-s}$,
which is known to be convergent and holomorphic on $\{\mathrm{Re} \ s >\frac n 2 \} $.
Following \cite{RS}, we
denote by $\det(\Omega)$ the regularized determinant of the Dirichlet Laplacian in $\Omega$ defined by 
\begin{equation}\label{detdef}
\det(\Omega)=\exp \left(-\zeta_{\Omega}'(0)\right).
\end{equation}

\smallskip

Eigenvalue optimization problems date from Rayleigh's ``Theory of Sound'' (1877), where it was suggested that the disk should minimize the first eigenvalue $\lambda_1$ among all 
planar domains of given measure. Rayleigh's conjecture was proved in any dimension
independently by Faber \cite{F} and Krahn \cite{Kr}.

\smallskip

Later, Luttinger \cite{Lut} proved an isoperimetric result analogous to 
Faber-Krahn for $Z(t)$, considered as a
functional 
on
the set of bounded Euclidean domains, that is, for any bounded domain $\Omega\subset
\mathbb{R}^n$ and any $t>0$,
Luttinger showed that
$$Z_{\Omega}(t)\le Z_{\Omega^*}(t),$$ 
where $\Omega^*$ is a Euclidean ball whose volume is equal to that of $\Omega$.  A similar property was proved in \cite{OPS} for the regularized determinant of the Laplacian in two dimensions (see \cite {LaMo, R} for other examples of  results in this direction).

\smallskip

The case of
multiply connected
planar domains, i.e. whose boundary admits more than one component, was first considered by Hersch. Using the method of interior parallels,  in  \cite{H2} Hersch proved  the following extremal property of annular membranes:\\ ``\emph{A doubly connected fixed membrane, bounded by two circles of given radii, has maximum $\lambda_1$ when the circles are concentric}''. 
 
 \smallskip
 
Hersch's result has been extended to a wider class of domains in any dimension by Harrell, Kr\"oger and Kurata \cite{HKK} and  Kesavan \cite{K}, 
whereby the authors consider a fixed domain $D$ from which an ``obstacle'' of fixed shape, usually spherical, has been
excised.   The position of the obstacle is allowed to vary, and the problem is to maximize or minimize $\lambda_1$. 
The critical assumption on the domain $D$ in  \cite{HKK} is an ``interior symmetry property,'' and with this assumption  
the authors further proved that, for the special case of two balls, $\lambda_1$ decreases when the center of
the small ball (the obstacle) moves away from the center of the large ball,
using a technique of domain reflection.  For a wider class of domains containing obstacles, it was shown
in \cite{HKK} that 
the maximizing position of the obstacle resides in a special subset of $D$, 
which in the case where $D$ is convex corresponds to what has later come to be called the {\em heart} of $D$ in \cite{BrMa, BrMaSa},
denoted
$\heartsuit(D)$ (see the definition below).
El Soufi and Kiwan \cite {EK1, EK2} have moreover proved other extensions of Hersch's result including one   valid for the second eigenvalue $\lambda_2$.  

\smallskip

The main aim of this paper is to establish a Hersch-type extremal result for the heat trace,
the spectral zeta function, and the determinant of the Laplacian, as well as
suitable generalizations for more general outer domains.  We begin by stating the special case of domains bounded by balls:
Given two positive 
numbers $R>r$ and a point ${\bf x}\in \R^n$, $|{\bf x}|<R-r$,  we denote by $\Omega({\bf x})$ the domain of $\R^n$ obtained by removing the ball  $B({\bf x},r)$ of radius $r$ centered at ${\bf x}$ from within the ball of radius $R$ centered at the origin. 

\begin{theorem}\label{ball}
{Let $\Omega({\bf x})$ be the domain bounded by balls as in the preceding paragraph.}

\smallskip

(i) For every  $t>0$, the heat trace $Z_{\Omega({\bf x})}(t)$ 
is nondecreasing
as the point ${\bf x}$ moves from the origin {directly} towards the boundary of the larger ball. 
In particular, $Z_{\Omega({\bf x})}(t)$ is minimal when
the balls are concentric (${\bf x}=O$) and maximal 
when the small ball is in contact with
the boundary of the larger ball ($|{\bf x}|=R-r$).

\smallskip

(ii)  For every  $s>0$, the zeta function   $\zeta_{\Omega({\bf x})}(s)$ increases as the point ${\bf x}$ moves from the origin {directly} towards the boundary of the larger ball. 
In particular, $\zeta_{\Omega({\bf x})}$ is minimal when the balls are concentric  and  maximal
in the limiting situation when the small ball approaches
the boundary of the larger ball.

\smallskip
(iii) The determinant of the Laplacian  $\det(\Omega({\bf x}))$ decreases as the point ${\bf x}$ moves from the origin {directly} towards the boundary of the larger ball. 
In particular, $\det(\Omega({\bf x}))$  is maximal when the balls are concentric  and  minimal
in the limiting situation when the small ball approaches
the boundary of the larger ball.

\end{theorem}

Let us clarify what we mean by the ``limiting situation when the ball approaches the boundary''
in this and in the more general Theorem \ref{heart}.
Our conclusion in these cases is derived by contradiction.  By assuming that the obstacle is in the interior of the domain, extremality can be excluded.  If the function under consideration is continuous, then, since the set of configurations including the cases where the obstacle touches the boundary is compact, the extremum is attained there.  Since the heat trace is continuous with respect to translations of the obstacle, claim (i) is not problematic, and the same is true for claim (ii) when $s > \frac n 2$.  For
smaller real values of $s$ 
the spectral zeta function is defined by analytic continuation, and the determinant of the Laplacian 
is defined by \eqref{detdef}.  In the proof these quantities will be shown to be continuous as a function of the position of the obstacle as it moves to the boundary, and their limits are what we describe as the ``limiting situations'' of the theorems.  The subtlety here is that when the obstacle is in contact with the boundary, a cusp is formed, as a consequence of which the heat-trace 
asymptotics are not necessarily known well enough
to allow a direct definition of $\zeta(s)$, $s\le\frac n2$, by analytic continuation, cf. \cite{Vas}.

Since $e^{- \lambda_1(\Omega) t}$ is the leading term in $Z_{\Omega}(t)$ as $t$ goes to infinity, it is clear that Theorem \ref{ball} implies the optimization result 
 mentioned above for $\lambda_1$. 

In order to state the more general theorem of which Theorem \ref{ball} is a special case, we recall some definitions.

\begin{definition} Let $P$ be a hyperplane in $\R^n$ which intersects $D$ so that  $D\setminus P$ is the union of two open subsets located on either side of $P$. 
According to \cite{HKK}, the domain $D$ is said to have the \emph{interior reflection property} with respect to  $P$ if
the reflection through $P$ of one of  these subsets, denoted $D_s$, is contained in $D$. 
Any such P will be called a 
\emph{hyperplane of
interior reflection} for $D$. 
The subdomain $D_s$ will be called the 
{\em small side} of $D$ (with respect to $P$) and {\bf its complement  }  $D_b=D\setminus D_s$  will be called the {\em big side}.

When $D$ is convex,  the \emph{heart} of $D$ is defined as the intersection of all the big sides with respect to
the hyperplanes of interior reflection of $D$. We denote it $\heartsuit(D)$.
\end{definition}
This definition  of  $\heartsuit(D)$ is equivalent to that introduced in  \cite{BrMa, BrMaSa},
where several properties of the heart of a convex domain are investigated.
A point ${\bf x}$ belongs to $\heartsuit(D)$ if either 
 there is 
no hyperplane of interior reflection passing through ${\bf x}$ or if any hyperplane of interior reflection passing through ${\bf x}$ is such that the reflection of $\partial D_s\setminus P$ touches $\partial D_b$. The first situation occurs when ${\bf x}$ is an interior point of $\heartsuit(D)$ while the latter is characteristic of the boundary points of $\heartsuit(D)$.

\smallskip

By construction the heart of a bounded convex
domain $D$ is a nonempty relatively
closed subset of $D$. Moreover, if $D$ is strictly convex and bounded, then
${\rm dist} (\heartsuit(D), \partial D)>0$.   We observe that for the ball and for many other domains with sufficient symmetry to identify an unambiguous center
point,
the heart is simply 
the  center.
It is, however, shown in \cite{BrMa}
that without reflection symmetries 
the typical heart has non-empty interior, even for simple polygons.  
For instance, for
an asymmetric acute triangle, it is a quadrilateral bounded by two angle bisectors and two perpendicular axes, while for an asymmetric obtuse triangle, it can be either a quadrilateral or a pentagon.

\begin{theorem}\label{heart} Let $D$ be a bounded
 $C^2$ convex domain of $\R^n$ and let $r>0$ be such that $D_r=\{{\bf x}\in D : {\rm dist}({\bf x}, \partial D) > r\}\ne \emptyset$.
For every ${\bf x}\in \bar D_r$ we set $\Omega ( {\bf x} ) =D\setminus B({\bf x},r)$. 

\smallskip
\noindent(i) 
For each fixed
$t>0$, the function ${\bf x}\in \bar D_r\mapsto Z_{\Omega({\bf x})}(t)$ achieves its minimum at a point
${\bf x}_0(t) \in \heartsuit(D) $,
while the maximum is achieved either  at 
an interior point
${\bf x}_1(t)$ of $\heartsuit(D)$ or 
in the limiting situation when the  ball approaches
the boundary of $D$.
\smallskip

\noindent(ii) For each fixed
$s>0$ and $ {\bf x} \in D_r\setminus  \heartsuit(D)$, the zeta function satisfies
\begin{equation}\label{supzeta}
\zeta_{\Omega({\bf x})}(s)>\inf_{{\bf y}\in  \heartsuit(D)\cap D_r} \zeta_{\Omega({\bf y})}(s),
\end{equation}
and  $\zeta_{\Omega({\bf x})}(s)$ is less than the supremum of all the values attained by the zeta function in the limiting situations when the  ball approaches
the boundary of $D$.
Moreover, if $r< {\rm dist}(\heartsuit(D), \partial D)$, then  $ {\bf x} \in D_r\mapsto \zeta_{\Omega({\bf x})}(s)$ achieves its  infimum at 
a point ${\bf x}_0(s) \in \heartsuit(D)$,
while the supremum is reached either  at 
an interior point
${\bf x}_1(s)$ of $\heartsuit(D)$ or 
in a limiting situation when the  ball approaches
the boundary of $D$.

\smallskip

\noindent(iii)
The regularized  determinant of the Laplacian  satisfies 
\begin{equation}\label{supdet}
\det (\Omega (  {\bf x}  ))<\sup_{{\bf y}\in  \heartsuit(D)\cap D_r}\det (\Omega (  {\bf y}  ))
\end{equation}
for every $ {\bf x} \in D_r\setminus  \heartsuit(D)$, 
and $\det (\Omega (  {\bf x}  ))$ is greater than the infimum of all the values attained by the determinant in the limiting situations when the  ball approaches
the boundary of $D$.
Moreover, if $r< {\rm dist}(\heartsuit(D), \partial D)$, then  the function $ {\bf x} \in D_r\mapsto \det (\Omega (  {\bf x}  ))$ achieves its supremum at 
a point 
${\bf x'}_0 \in \heartsuit(D)$, while 
 the infimum is achieved either  at  an interior point ${\bf x'}_1$ of $\heartsuit(D)$ or in a limiting situation when the  ball approaches the boundary of $D$.

\end{theorem}

We remark that a straightforward consideration of the limit $t \to \infty$
leads back to related results of
 \cite{HKK}.  As in that article, it is not difficult to extend Theorem \ref{heart} to many nonconvex
 domains, at the price of entering into the sometimes complex nature of 
$\heartsuit(D)$.  For simplicity we limit the present article to the case of convex $D$.

\smallskip

We conjecture that
 the minimum of $Z_{\Omega({\bf x})}(t)$
is never achieved outside $\heartsuit(D)$, and that
 the maximum of $Z_{\Omega({\bf x})}(t)$ and
resp. the maximum of $\zeta_{\Omega({\bf x})}(s)$ and  the minimum of $\det (\Omega (  {\bf x}  ))$), are
achieved only 
in the limiting situation where  $B({\bf x},r)$ touches the boundary of  $D$. This is certainly
the case for example when a convex domain  $D$ admits a hyperplane of symmetry since then, ${\rm int }(\heartsuit(D))=\emptyset$. 

\begin{remark} 
We shall approach the analysis of the spectral zeta function and the regularized determinant  through  order-preserving integral transforms relating them to the heat trace.  As in \cite{HaHe}, transform theory can 
similarly be used to obtain
corollaries for many 
further functions, e.g., Riesz means, with respect to the optimal position of an obstacle.


 \end{remark} 
A main ingredient of the proof is the following Hadamard-type formula for the first variation of $Z_\Omega(t)$ with respect to a deformation $\Omega_\varepsilon=f_\varepsilon(\Omega)$ of the domain : 
\begin{equation}\label{hadamard-heat}
\frac{\partial}{\partial \varepsilon} Z_{\Omega_{\varepsilon}}(t)\big|_ {\varepsilon=0} = -{t\over
2} \int_{\partial\Omega}  \Delta K(t,{\bf x},{\bf x})  v({\bf x}) dx,
\end{equation}
where $v= X\cdot \nu$ is the   component of the deformation vectorfield $X=\frac{df_\varepsilon}{d\varepsilon}\vert_{\varepsilon=0}$ in the direction of the inward unit normal $\nu$, and  $K$ is the heat kernel
(cf. \cite[Theorem 4.1]{EI1}).   Notice that this formula coincides with that given by Ozawa in \cite[Theorem 4]{oza} for deformations of the form $f_\varepsilon({\bf x}) ={\bf x} +\varepsilon \rho ({\bf x})\nu({\bf x})$ along the boundary, where $\rho$ is a smooth function on $\partial\Omega$. Indeed, it is easy to check that for all $ {\bf x}\in\partial\Omega$, $v({\bf x})=\rho ({\bf x}) $ and (using \eqref{heatkerseries} below)
$ \Delta K(t,{\bf x},{\bf x}) = {2} \sum_{k=0}^\infty e^{-\lambda_k t} \vert\nabla u_k({\bf x})\vert^2= {2} \sum_{k=0}^\infty e^{-\lambda_k t} \vert\frac{\partial u_k}{\partial\nu}({\bf x})\vert^2$. For more information about Hadamard deformations we refer to 
\cite{EI2, EI1, Gar, GaSc, Henry, oza, RS}.)
 \section {Proof of results}\label{2}

 Let $\Omega\subset \R^n$ be a   domain  of the form $\Omega=D\setminus B$, where $D$ is a bounded domain and $B$ is a convex domain  such that the closure of $B$ is contained in $D$. 
{(For simplicity our theorems have been stated for the case of a spherical obstacle $B$, but the essential argument requires only a lower degree of symmetry.)}

\smallskip
Let us start by establishing how the zeta function and the Laplacian determinant  are related to the heat trace in our situation. Indeed, the following formula  is valid for every complex number $s$ with $\mathrm{Re} \ s>\frac n2$:
$$ \zeta_{\Omega} (s):=\sum_{k \geq 1}{}\ {1\over{\lambda_k^s(\Omega})}=
{1\over \Gamma(s)}
\int_{0}^{\infty} Z_{\Omega}(t) t^{s-1}dt.$$
It is well known that the function $Z_{\Omega}$ satisfies
\begin{equation}\label{ZetaAsym}
Z_{\Omega}(t)\sim \sum_{k \geq 0}
a_{k}\ t^{(k-n) \over 2}\qquad {\rm as }\  t\to 0,
\end{equation}
where $a_k$ is a sequence of real numbers that depend only
on the geometry of the boundary of $\Omega$ (see e.g. \cite{BrGi}). In particular, {\bf the 
coefficients $a_k$ 
are independent of the position of $B$ within $D$}.
We set
$$ \tilde Z_{\Omega}(t) =  Z_{\Omega}(t) - \sum_{k =0 }^{n}
a_{k}\ t^{(k-n) \over 2},$$ 
so that $\tilde Z_{\Omega}(t)/\sqrt  t $ is a bounded function in 
a neighborhood of $t=0$. We also introduce the meromorphic  function 
$$
R(s)= {1\over \Gamma(s)}\sum_{k =0 }^{n}a_{k}\int_{0}^{1} t^{s-1+(k-n)/2}dt = {1\over \Gamma(s)}\sum_{k =0 }^{n}
\frac{a_{k}}{s-(n-k)/ 2},
$$
which has poles at $1/2, 1, 3/ 2, 2, \cdots,  n/2$.  (Note that $s=0$ is not a pole since $1\over s\Gamma(s)$ is a holomorphic function on $\C$.)  
Consequently,
for every 
$s\in \R^+$,
\begin{equation}\label{zeta}
  \zeta_{\Omega} (s)= R(s)+
{1\over \Gamma(s)} 
\int_{0}^{1} \tilde Z_{\Omega}(t) t^{s-1}dt + {1\over \Gamma(s)} 
\int_{1}^{\infty} Z_{\Omega}(t) t^{s-1}dt.
\end{equation}
where the last term is an entire function of $s$, since $ Z_{\Omega}(t) $ behaves as $ e^{-\lambda_1(\Omega) t}$ when $t\to +\infty$.  
On the other hand, the reciprocal gamma function $f(s):={1\over \Gamma(s)}$ vanishes at $s=0$ and satisfies $f'(0)=1$. Therefore,
\begin{equation}\label{zetaprime}
\zeta'_{\Omega} (0)= R'(0) + 
\int_{0}^{1} \tilde Z_{\Omega}(t) t^{-1}dt + 
\int_{1}^{\infty} Z_{\Omega}(t) t^{-1}dt .
\end{equation}

\smallskip

Assume that the domain $D$  has the interior reflection property with respect
to a hyperplane $P$ about which the set $B$ is reflection-symmetric. 
(Here we do not need to restrict to convex $D$.)
Our strategy is to consider a
displacement of the obstacle by $\varepsilon$ in a certain direction and to
show that $Z_{\Omega_{\varepsilon}}(t)$ is monotonically increasing in that direction.

\smallskip
Thus let $V$ be the unit vector perpendicular to $P$ and pointing in the direction of the small side  $D_s$. For small $\varepsilon >0$, we translate $B$ 
by a distance $\varepsilon$ in the direction of $V$ and set $B_{\varepsilon}:=B+\varepsilon V$ 
and $\Omega_{\varepsilon}:=D\setminus B_{\varepsilon}$.   The results of this paper rely on the following proposition.

 \begin{proposition}\label{det} 
Assume that the domain $D$  has the interior reflection property with respect
to a hyperplane $P$ about which the set $B$ is reflection-symmetric. Consider displacements 
as described above.
Then,
\begin{equation}\label{derivZ}
\frac{\partial}{\partial \varepsilon} Z_{\Omega_{\varepsilon}}(t)\big|_ {\varepsilon=0} >0,
\end{equation}
{except possibly for a finite set of values $t$
in any interval $[\tau,\infty)$ with $\tau > 0$.  Moreover, for each $s > 0$,}

\begin{equation}\label{derivzeta}
\frac{d}{d \varepsilon} \zeta_{\Omega_{\varepsilon}}(s)\big|_ {\varepsilon=0} >0,
\end{equation}
and
\begin{equation}\label{derivdet}
\frac{d}{d \varepsilon} \det({\Omega_{\varepsilon}})\big|_ {\varepsilon=0} <0.
\end{equation}
\end{proposition}

 \begin{proof} The heat kernel $K$ on $\Omega$ under the
Dirichlet boundary condition is defined as the fundamental solution of the heat equation, that is

$$\left\{
\begin{array}{l}
({\partial\over \partial t } - \Delta_y) K(t,{\bf x},{\bf y}) =0 \;\; \hbox{in}\; \Omega\\
\\
K(0^+,{\bf x},{\bf y})=\delta_{\bf x}({\bf y})\\
\\
K(t,{\bf x},{\bf y}) =0 \;\; \forall {\bf y}\in \partial \Omega. \;
\end{array}
\right.$$
The relationship between the heat kernel and the spectral decomposition of the Dirichlet
Laplacian in $\Omega$ is given by
\begin{equation}\label{heatkerseries}
K(t,{\bf x},{\bf y})=\sum_{k\ge 1} e^{- \lambda_k(\Omega) t} u_k ({\bf x}) u_k ({\bf y}),
\end{equation}
where $(u_k)_{k\ge1}$ is an $L^2(\Omega)$-orthonormal family of eigenfunctions satisfying
$$\left\{
\begin{array}{l}
-\Delta u_k=\lambda_k (\Omega) u_k\;\; \hbox{in}\; \Omega\\
\\
u_k =0 \;\; \hbox{on}\; \partial \Omega.\\
\end{array}
\right.$$\\
The heat trace is then given by
$$Z_{\Omega}(t) = \int_\Omega K(t,{\bf x},{\bf x}) dx= \sum_{k\ge 1} e^{- \lambda_k (\Omega) t}.$$

Let $X$ be a smooth vectorfield such that $X$ vanishes on $\partial D$ and coincides with the vector $V$ on $\partial B$. For sufficiently small $\varepsilon$, one has
$\Omega_{\varepsilon} = f_\varepsilon (\Omega)$ where $f_\varepsilon ({\bf x})={\bf x}+\varepsilon X({\bf x})$.
The Hadamard-type formula \eqref{hadamard-heat} gives :
\begin{eqnarray*}
\frac{\partial}{\partial \varepsilon} Z_{\Omega_{\varepsilon}}(t)\big|_ {\varepsilon=0} &=& -{t\over
2} \int_{\partial\Omega}   \Delta K(t,{\bf x},{\bf x}) \left(X\cdot \nu\right)({\bf x}) dx\\
&=& -{t\over
2} \int_{\partial B}  \Delta K(t,{\bf x},{\bf x}) \, V\cdot\nu ({\bf x}) dx.
\end{eqnarray*}
Let $B_s$ be the half of $B$ contained in the small side $D_s$ of $D$ and $(\partial B)_s=\partial B\cap D_s$.  
Here, we assume without loss of generality  that $D_s$ is connected  (otherwise,  $B_s$ is contained in one connected component of $D_s$ and we concentrate our analysis on this single component).
Using the symmetry assumption on $B$ with respect to $P$ we obtain
\begin{equation}\label{hadamard}
\frac{\partial}{\partial \varepsilon} Z_{\Omega_{\varepsilon}}(t)\big|_ {\varepsilon=0} = -{t\over
2} \int_{(\partial B)_s}   \left(\Delta K(t,{\bf x},{\bf x})-\Delta K(t,{\bf x}^*,{\bf x}^*)\right) \, V\cdot\nu ({\bf x}) dx
\end{equation}
where ${\bf x}^*$ stands for the reflection of ${\bf x}$ through $P$. 

\medskip
Define the function $\phi(t,{\bf x},{\bf y})=K(t,{\bf x},{\bf y})-K(t,{\bf x}^*,{\bf y}^*)$ on $(0,\infty)\times \Omega_s\times\Omega_s$ with $\Omega_s=D_s\setminus B_s$. 

\medskip

\textit{Claim }: For all $(t,{\bf x},{\bf y})\in (0,\infty)\times \Omega_s\times\Omega_s$, $\phi(t,{\bf x},{\bf y}) \le 0$.

\noindent
Let us check the sign of $\phi(t,{\bf x},{\bf y})$ on  $(0,\infty)\times \partial\Omega_s\times\partial \Omega_s$. Notice that $\partial\Omega_s$ is the union of three components : $(\partial D)_s$, $(\partial B)_s$ and $\Omega \cap P$.  First, 
from the boundary conditions, if ${\bf x}\in (\partial B)_s$ or ${\bf y}\in (\partial B)_s$, then $K(t,{\bf x},{\bf y})=K(t,{\bf x}^*,{\bf y}^*)=0$ and, hence, $\phi(t,{\bf x},{\bf y}) =0$. On the other hand, $K(t,{\bf x},{\bf y})$ vanishes as soon as ${\bf x}\in (\partial D)_s$ or ${\bf y}\in (\partial D)_s$, which implies $ \phi(t,{\bf x},{\bf y})=-K(t,{\bf x}^*,{\bf y}^*)\le 0$.  It remains to consider the case where both ${\bf x}$ and ${\bf y}$ belong to  $\Omega \cap P$. In this case we have ${\bf x}^*={\bf x}$, ${\bf y}^*={\bf y}$ and $\phi(t,{\bf x},{\bf y}) =0$.

Observe next that for all ${\bf x}\in \bar\Omega_s$, the function $(t,{\bf y})\mapsto \phi(t,{\bf x},{\bf y})$ is a solution of the following parabolic problem :
$$(*)\left\{
\begin{array}{l}
({\partial\over \partial t } - \Delta_y) \phi(t,{\bf x},{\bf y}) =0 \;\; \hbox{in}\; \Omega_s\\
\\
\phi(0^+,{\bf x},{\bf y})=0.
\end{array}
\right.$$
Given any ${\bf x}\in\partial\Omega_s$, the parabolic maximum principle (see e.g., \cite{Eva}, 
{\S 7.1})
tells us that, since $(t,{\bf y})\mapsto \phi(t,{\bf x},{\bf y})$  is nonpositive on the boundary of the cylinder $(0,\infty)\times \Omega_s$, 
it follows that
$ \phi(t,{\bf x},{\bf y})\le 0$ for all $t>0$ and all ${\bf y}\in \bar\Omega_s$.

Now,  from the symmetry of $\phi$ with respect to ${\bf x}$ and ${\bf y}$, the function $(t,{\bf x})\mapsto \phi(t,{\bf x},{\bf y})$ satisfies the same parabolic system as $(*)$. Since we 
have established
that $\forall {\bf y}\in \bar\Omega_s$, the function $(t,{\bf x})\mapsto \phi(t,{\bf x},{\bf y})$ is 
everywhere nonpositive on the boundary of the cylinder $(0,\infty)\times \Omega_s$, the parabolic maximum principle then implies 
that   
$\phi(t,{\bf x},{\bf y}) $ is nonpositive in the whole cylinder $(0,\infty)\times \Omega_s\times\Omega_s$.

\medskip

\textit{Claim }:  $\Delta \phi (t,{\bf x},{\bf x})\le0$ for all $(t,{\bf x})\in (0,\infty)\times (\partial B)_s$.

For a sufficiently small $\delta >0$, let $V=\{\psi({\bf z},\rho):={\bf z}+\rho\ \nu({\bf z}) \; ; \; {\bf z}\in\partial B\mbox{ and }  0\le\rho < \delta\}$ be the 1-sided $\delta$-tubular neighborhood of $\partial B$.  The Euclidean metric $g_{E}$ can be expressed in $V$ with respect to so-called Fermi coordinates $({\bf z},\rho)\in\partial B \times (0,\delta)$ as follows (see for instance \cite[Lemma 3.1]{pac}) :
$$g_E= d\rho^2+g_\rho,$$
where $g_\rho $ is a Riemannian metric on the hypersurface $ \Gamma_\rho= \{{\bf z}+\rho\ \nu({\bf z}) \; ; \; {\bf z}\in\partial B\}$. Consequently, the Euclidean 
 Laplacian in $V$ takes on the following form  with respect to Fermi coordinates :
 $$\Delta=  \frac{\partial^2}{\partial \rho^2} - H_\rho   \frac{\partial}{\partial \rho} + \Delta_{g_\rho},$$
where $H_\rho$ is the mean curvature of $ \Gamma_\rho$ and $\Delta_{g_\rho}$ is the Laplace-Beltrami operator of $(\Gamma_\rho, g_\rho)$.

Now, $ K(t,{\bf x},{\bf x})=\sum_{k\ge 1} e^{- \lambda_k(\Omega) t} u_k ({\bf x})^2$, 
and it is known that for $C^2$ domains, $\|\nabla u_k\|_\infty$ 
is bounded by a constant times a finite power of 
$\lambda_k$  (see \cite{Gri,HaTa}).  
{Hence, the functions $ K(t,{\bf x},{\bf x})$ and 
consequently $\phi (t,{\bf x},{\bf x})$ vanish} quadratically  on $(\partial B)_s$. Thus,  for any point $ {\bf z}=\psi({\bf z}, 0)   \in (\partial B)_s $,
$$ \frac{\partial }{\partial \rho}\phi (t,\psi({\bf z}, \rho),\psi({\bf z}, \rho))  \big\vert_{\rho=0}=0 \quad \mbox{and}\quad \Delta_{g_\rho}\phi (t,\psi({\bf z}, \rho),\psi({\bf z}, \rho)) \big\vert_{\rho=0} =0.$$
Therefore,
$$
\Delta \phi (t,{\bf z},{\bf z})= \frac{\partial^2}{\partial \rho^2}\phi (t,\psi({\bf z}, \rho),\psi({\bf z}, \rho))\big\vert_{\rho=0},
$$
which is nonpositive since $ \frac{\partial }{\partial \rho}\phi (t,\psi({\bf z}, \rho),\psi({\bf z}, \rho))  \big\vert_{\rho=0}=0 $ and, according to what we proved in  the previous claim, the function  $\rho\in[0,\delta)\mapsto  \phi (t,\psi({\bf z}, \rho),\psi({\bf z}, \rho))$ achieves its maximum at $\rho=0$.

\medskip

\textit{Claim }: Let $\tau $ be any positive  real number. Except possibly for a finite set of values of $t$ 
in $[\tau,\infty)$,
$$
\frac{\partial}{\partial \varepsilon} Z_{\Omega_{\varepsilon}}(t)\big|_ {\varepsilon=0} >0.$$
From the assumptions that $B$ is convex and symmetric with respect to the hyperplane $P$, it follows that the product $V\cdot\nu({\bf x})$ is positive at
{almost}
every point ${\bf x}$ of  $(\partial B)_s$. From equation \eqref{hadamard} and the previous claim we then deduce that $\forall t>0$, 
$$
\frac{\partial}{\partial \varepsilon} Z_{\Omega_{\varepsilon}}(t)\big|_ {\varepsilon=0} \ge 0.
$$
 
To show that this quantity cannot vanish at more than a finite set of values of $t\in [\tau,\infty)$, we shall show
that it is analytic as a function of $t$ in the open right half plane, and
positive for real values of $t$ sufficiently large.  By the unique continuation theorem an 
analytic function that vanishes on a set with a point of accumulation is identically zero,
which would pose a contradiction.

To establish
{the analytic properties}
of $Z_{\Omega_{\varepsilon}}(t)$, we
argue as follows.
Observe first that the deformation $\Omega_\varepsilon$ depends analytically on $\varepsilon$
{in a neighborhood of $0$},
since $\Omega_\varepsilon=f_\varepsilon (\Omega)$ with $f_\varepsilon ({\bf x})={\bf x}+\varepsilon X({\bf x})$. As in the proof of Lemma 3.1 in \cite{EI1}, the
{Dirichlet Laplacian in $\Omega_\varepsilon$  is an  analytic family of operators in the sense of 
Kato
\cite{Kat} with 
respect to the parameter $\varepsilon$.}
Because each eigenvalue of the Laplacian is at most finitely degenerate, 
according to
{\cite[p. 425]{Kat},}
there is a numbering of the eigenvalues 
$\{\lambda_k(\Omega_\varepsilon)\} \to \{\Lambda_k(\varepsilon)\}$ for which each
$ \Lambda_k(\varepsilon)$ is analytic in $\varepsilon$ in a neighborhood of 
$\varepsilon=0$. 
(Using this numbering, which is important only in a neighborhood
of a degenerate eigenvalue,
does not alter $Z(t)$ as defined in
\eqref{Zdef}.)
In consequence of the Hadamard formula for the derivative of an eigenvalue,
$\frac{\partial \Lambda_k}{\partial \varepsilon} |_{\varepsilon=0}$
is dominated 
in norm by the integral of the square of the normal derivative of an
associated $L^2$ normalized eigenfunction
$u_k$ over the boundary of the obstacle.
{We again call upon estimates for $C^2$ domains, by which both}
$\|u_k\|_\infty$ and $\|\nabla u_k\|_\infty$ are bounded by constants times finite powers of 
$\lambda_k$ 
\cite{Gri,HaTa}, which in turn $\sim k^{\frac{2}{n}}$ by the Weyl law.  It follows that 
both the series $\sum_{k\ge 1} e^{- \lambda_k(\Omega_{\varepsilon}) t}$ and its term-by-term derivative 
with respect to $\varepsilon$ converge uniformly on each set of the 
form $\{{\rm Re  }\  t \ge \tau > 0\}$, 
and are therefore analytic on such sets.

To finish the argument, we observe that by differentiating $ Z_{\Omega_{\varepsilon}}(t) = \sum_{k\ge 1} e^{- \lambda_k(\Omega_{\varepsilon}) t}$,
$$
\left.\frac{\partial}{\partial \varepsilon} Z_{\Omega_{\varepsilon}}(t)\right|_ {\varepsilon=0}
= e^{- \lambda_1(\Omega) t}\left( - \left.\frac{\partial \lambda_1(\Omega_\varepsilon)}{\partial \varepsilon}\right|_ {\varepsilon=0} 
+  0(e^{- (\lambda_2 -\lambda_1)t})\right).
$$
This is  positive for large $t$ because $\lambda_1$ is nondegenerate and
$\frac{\partial \lambda_1(\Omega_\varepsilon)}{\partial \varepsilon}\big|_ {\varepsilon=0} < 0$  by 
\cite{HKK}.

\medskip
This completes the proof of \eqref{derivZ}.  
 The proof of {\eqref{derivzeta} and} \eqref{derivdet} relies on the  formulae \eqref{zeta} and \eqref{zetaprime} that give  for every 
$s\in \R^+$ and every $\varepsilon\ne 0$ sufficiently small,
\begin{equation}\label{zeta_epsilon}
  \zeta_{\Omega_\varepsilon} (s)= R(s)+
{1\over \Gamma(s)} 
\int_{0}^{1} \tilde Z_{\Omega_\varepsilon}(t) t^{s-1}dt + {1\over \Gamma(s)} 
\int_{1}^{\infty} Z_{\Omega_\varepsilon}(t) t^{s-1}dt
\end{equation}
and
\begin{equation}\label{zetaprime_epsilon}
\zeta'_{\Omega_\varepsilon} (0)= R'(0) + 
\int_{0}^{1} \tilde Z_{\Omega_\varepsilon}(t) t^{-1}dt + 
\int_{1}^{\infty} Z_{\Omega_\varepsilon}(t) t^{-1}dt ,
\end{equation}
with $\det (\Omega_\varepsilon ) = e^{-\zeta'_{\Omega_\varepsilon} (0)}$.
\end{proof} 

\begin{proof}[Proof of Theorem \ref{heart}] Let $D$ be a bounded convex domain of $\R^n$  
and let $r>0$ be less than the inradius of $D$.
Observe first that for every $t>0$,  the function ${\bf x}\mapsto Z_{\Omega({\bf x})}(t)$ is continuous on  $ \bar D_r=\{ {\bf x} \in D : {\rm dist}({\bf x}, \partial D) \ge r \} $. 
{Indeed, we know that 
$$\lambda_k(\Omega({\bf x}))\ge \frac {n}{n+2} C_n \left(\frac k{\vert\Omega({\bf x})\vert }\right)^{\frac2n}$$
(see \cite{LY}), where}
$C_n$ is  the constant appearing in Weyl's asymptotic formula and $\vert\Omega({\bf x})\vert$ is  the volume  of $\Omega({\bf x})$. Since $\vert\Omega({\bf x})\vert$ does not depend on ${\bf x}$, we deduce that the series $\sum e^{-\lambda_k(\Omega({\bf x})) t}$ converges uniformly on $ \bar D_r$ and that its sum $Z_{\Omega({\bf x})}(t)$ depends continuously on ${\bf x}$. (The continuity of eigenvalues $\lambda_k(\Omega({\bf x})) $ can be derived in several ways from standard continuity results cited in \cite[Section 2.3.3]{He}.  In particular, see Remark 6.2 of  \cite{Dan}.) Consequently, $Z_{\Omega({\bf x})}(t)$ achieves its extremal values in $  \bar D_r$. 

 \smallskip
Let ${\bf x}\in D_r =\{ {\bf x} \in D : {\rm dist}({\bf x}, \partial D) > r \} $   be a point such that ${\bf x} \notin \heartsuit(D)$. From the definition of $\heartsuit(D)$, 
there exists a hyperplane of
interior reflection $P$ of $D$ passing through ${\bf x}$.
Moreover, since the reflection of $\partial D_s\setminus P$ is disjoint from $\partial D_b$, there exists $\delta>0$ such that $\forall \varepsilon\in[0,\delta]$,  the hyperplane $P_\varepsilon$ parallel to $P$ and passing through ${\bf x_\varepsilon} = {\bf x}-\varepsilon V$ is a hyperplane of
interior reflection, where $V$ is the unit vector perpendicular to $P$ and pointing in the direction of $D_s$. Applying Proposition \ref{det}, we see that  the function $\varepsilon\mapsto Z_{\Omega({\bf x_\varepsilon})}(t) $ is monotonically nonincreasing 
(notice that the variation formula \eqref{derivZ} is given for a displacement into the small side $D_s$. Here,  the obstacle moves in the opposite direction, that of  $-V$, which has the effect of changing the sign of the derivative.) At the same time, the distance ${\rm dist}({\bf x}_\varepsilon,\heartsuit(D) ) $ is clearly decreasing since ${\bf x}_\varepsilon$ moves  into the big side. 
It follows that the set of points where ${\bf x}\mapsto Z_{\Omega({\bf x})}(t)$ achieves its minimum cannot be disjoint from  $\heartsuit(D)$.
The minimum is therefore achieved at a point  ${\bf x}_0(t)\in \heartsuit(D)$.

\smallskip
Similarly,  if a point ${\bf x}\in D_r $  does not belong to the interior of $\heartsuit(D)$, then there exists a hyperplane of
interior reflection passing through ${\bf x}$ so that it is possible to move  the obstacle  into the small side $D_s$ along a line segment perpendicular to $P$.
The function $Z_{\Omega({\bf x})}(t) $ is monotonically nondecreasing along such  displacement while the obstacle approaches the boundary of $D$. Again, this proves that if the set of points where ${\bf x}\mapsto Z_{\Omega({\bf x})}(t)$ achieves its maximum is not contained in the interior of   $\heartsuit(D)$, then it must hit $\{{\bf x}\in D\ :\  {\rm dist}({\bf x}, \partial D)=r\}$.

\medskip
  

\smallskip
The continuity of the zeta function and of the determinant in $D_r$ derive from the continuity of the heat trace, through  \eqref{zeta} and \eqref{zetaprime}, according to which, for every 
$s\in \R^+$ and every
${\bf x}\in D_r$,
\begin{equation}\label{zeta_x}
  \zeta_{\Omega({\bf x})} (s)= R(s)+
{1\over \Gamma(s)} 
\int_{0}^{1} \tilde Z_{{\Omega({\bf x})}}(t) t^{s-1}dt + {1\over \Gamma(s)} 
\int_{1}^{\infty} Z_{{\Omega({\bf x})}}(t) t^{s-1}dt,
\end{equation}
and
\begin{equation}\label{zetaprime_x}
\zeta'_{{\Omega({\bf x})}} (0)= R'(0) + 
\int_{0}^{1} \tilde Z_{{\Omega({\bf x})}}(t) t^{-1}dt + 
\int_{1}^{\infty} Z_{{\Omega({\bf x})}}(t) t^{-1}dt,
\end{equation}
where $R(s)$ is a function that does not depend on ${\bf x}$. 
These formulae are not necessarily valid in the situation where the ball $B({\bf x},r)$ touches the boundary of $D$ since, due to the cuspidal singularity that the domain $\Omega({\bf x})$ will then present, in consequence of which
the function  $\tilde Z_{\Omega({\bf x})}(t)/\sqrt  t $ may cease to be bounded in the neighborhood of $t=0$. 

\smallskip
Let ${\bf x}\in D_r $ be a point lying outside  $\heartsuit(D)$ and let $s\in\R^+$. As before, using the definition of  $\heartsuit(D)$ and Proposition \ref{det}, we see that it is possible to move $B({\bf x},r) $ towards the heart so as to decrease  $ \zeta_{\Omega({\bf x})} (s)$. This enables us to construct a (possibly finite) sequence of points converging to a point ${\bf y}\in \heartsuit(D)$, along which the zeta function is decreasing. Thus, $ \zeta_{\Omega({\bf x})} (s)<  \zeta_{\Omega({\bf y})} (s)$, which  leads to \eqref{supzeta}. Similarly, it is possible to move $B({\bf x},r) $ in the direction of the boundary so as to increase $ \zeta_{\Omega({\bf x})} (s)$, which implies that $ \zeta_{\Omega({\bf x})} (s)$ is less than a  limiting  value of $ \zeta_{\Omega({\bf y})} (s)$ as $B({\bf y},r) $ approaches the boundary of $D$.\\ 
Now, when $r< {\rm dist}(\heartsuit(D),\partial D)$, the heart is  contained in $D_r$ and, consequently, ${\bf x}\mapsto \zeta_{\Omega({\bf x})} (s)$ is continuous on $\heartsuit(D)$ (which is a compact set) and achieves its minimum at a point ${\bf x}_0 (s)\in \heartsuit(D)$. Similarly, ${\bf x}\in\heartsuit(D) \mapsto \zeta_{\Omega({\bf x})} (s)$ achieves its maximum at a point which belongs to  the interior of $\heartsuit(D)$, since a ball of radius $r< {\rm dist}(\heartsuit(D),\partial D)$ centered at the boundary of $\heartsuit(D)$ can be moved away from $\heartsuit(D)$  so as to increase $\zeta$. 

\smallskip
The statement concerning the determinant can be proved using the same arguments.

\end{proof} 

\bigskip
{\bf Acknowledgments}  The authors gratefully 
note that much of this work was done
at the Centro de Giorgi in Pisa and while the second author was a 
visiting professor at the
Universit\'e Fran\c{c}ois Rabelais, Tours, France.
We also wish to thank the referee for pertinent comments.



\def\cprime{$'$}

\end{document}